\documentclass{amsart}
\usepackage{anyfontsize}
\usepackage{stmaryrd}
\SetSymbolFont{stmry}{bold}{U}{stmry}{m}{n}

\usepackage{amsmath,amsthm,amssymb}
\usepackage{braket}
\usepackage[english]{babel}
\usepackage{rotating}
\usepackage[textsize=tiny]{todonotes}

\usepackage[autostyle]{csquotes}
\usepackage[backend=biber,style=alphabetic,url=false,doi=true,eprint=false]{biblatex}
\newtheorem{thm}{Theorem}[section]
\newtheorem{prop}[thm]{Proposition}
\newtheorem{lemma}[thm]{Lemma}
\theoremstyle{definition}
\newtheorem{defin}[thm]{Definition}
\newtheorem{guideline}{Guideline}

\newcommand\power{\mathop{\mathcal P}}
\newcommand\scott[1]{\left\llbracket #1 \right\rrbracket}
\newcommand\dom{\mathop\mathrm{dom}}
\newcommand\FV{\mathop\mathrm{fv}}
\newcommand\BV{\mathop\mathrm{bv}}
\newcommand\Br{\mathop\mathrm{Br}}
\newcommand\fv\FV
\newcommand\rel{\mathop\text{rel}}
\newcommand\ESO{\mathrm{ESO}}
\newcommand\nmodels\nvDash
\renewcommand\models\vDash
\newcommand\rest{\mathbin\upharpoonright}
\newcommand\FO{\mathrm{FO}}
\newcommand\dep{\mathrm{d}}
\newcommand\mt{\mathrm{mt}}
\newcommand\M{\mathbb{M}}
\newcommand\depat{D}
\newcommand\indep{\mathbin \bot}
\newcommand\mmodels{\models_\mt}

\newcommand\truth\top

\makeatletter
\newcommand{\iland}{\mathrel{\vphantom{\land}\mathpalette\iland@\relax}}
\newcommand{\ilor}{\mathrel{\vphantom{\lor}\mathpalette\ilor@\relax}}
\newcommand{\iland@}[2]{\ooalign{\raisebox{.15ex}{\rotatebox[origin=c]{-90}{$\m@th#1\leqslant$}}}}
\newcommand{\ilor@}[2]{\ooalign{\raisebox{.15ex}{\rotatebox[origin=c]{-90}{$\m@th#1\geqslant$}}}}
\makeatother

\title[Generalized quantifiers using team semantics]{Generalized quantifiers using team semantics} 

\author{Fredrik Engstr\"om}
\email{fredrik.engstrom@gu.se}
\address{Department of Philosophy, Linguistics and Theory of Science, University of Gothenburg, Box 100, 405 30 Gothenburg, Sweden}


\begin{document}
\begin{abstract}
Dependence logic provides an elegant approach for introducing dependencies between
variables into the object language of first-order logic. In \cite{engstrom2012generalized}
generalized quantifiers were introduced in this context. However, a satisfactory account was only achieved for monotone increasing generalized quantifiers.

In this paper, we modify the fundamental semantical guideline of dependence logic to create a framework that adequately handles both monotone and non-monotone generalized quantifiers. We demonstrate that this new logic can interpret dependence logic and possesses the same expressive power as existential second-order logic (ESO) on the level of formulas. Additionally, we establish truth conditions for generalized quantifiers and prove that the extended logic remains conservative over first-order logic with generalized quantifiers and is able to express the branching of continuous generalized quantifiers.
\end{abstract}

\maketitle

%\listoftodos

\section{Introduction}

Dependence logic \cite{Vaananen:2007}  extends first-order logic by dependence
atoms of the form
\begin{equation*}
\depat(t_1,\ldots,t_n).
\end{equation*} 
These atoms express that the value of the term $t_n$ is functionally
determined by the values of $t_1,\ldots, t_{n-1}$. 

While in first-order logic, the order of quantifiers alone determines the
dependence relations between variables, dependence logic allows for more
general and non-linear dependencies between variables. Remarkably, dependence
logic is equivalent in expressive power to existential second-order logic
($\ESO$). Historically, it was preceded by the partially ordered quantifiers
(known as Henkin quantifiers) introduced by Henkin \cite{Henkin:1961}, as
well as the Independence-Friendly (IF) logic developed by Hintikka and
Sandu \cite{Hintikka:1989}.

The semantical framework of dependence logic, known as team semantics, has
proven to be highly flexible. It accommodates not only dependence atoms, but also other intriguing generalizations. For instance, variants of dependence atoms were introduced and explored in works such as \cite{engstrom2012generalized}, \cite{gradel2012dependence} and \cite{galliani2012inclusion}.

In \cite{engstrom2012generalized}, I investigated extensions of dependence logic using generalized quantifiers. I introduced a general schema for extending dependence logic with these quantifiers. Subsequently, this schema was further examined in subsequent works, including \cite{engstrom2013characterizing} and
\cite{engstrom2013dependence}. This extension allows us to express branching
 behavior of generalized quantifiers--a construct that naturally arises in
 natural language \cite{barwise1979branching}. Moreover, it provides a means
 to express different scope readings of generalized quantifiers in a natural
 and intuitive manner, as exemplified in \cite{aloni2022non}.

In the present paper, we take a slightly different approach and redefine the semantics of dependence logic to incorporate non-monotone generalized quantifiers into the logic. The traditional semantics of dependence logic relies on the following guiding principle. 

In the present paper we take a slightly different approach and redefine the
semantics of dependence logic to be able to incorporate non-monotone
generalized quantifiers into the logic. The traditional semantics of
dependence logic is based on the following guiding principle, or guideline.
\begin{guideline}
\label{g1}
A formula $\varphi$ is satisfied by a team $X$ if for
every assignment $s : \dom(X) \to M^k$, if $s \in X$ then $s$ satisfies
$\varphi$.
\end{guideline} This principle establishes the semantics for all formulas
 where it makes sense to say that a single assignment satisfies the formula,
 i.e., for all formulas that do not contain the dependence atom. From this
 principle, the semantical clauses of dependence logic naturally follow.

In this paper we will instead base our semantics on the following alternative
guideline.
\begin{guideline}
\label{g2}
A formula $\varphi$ is satisfied by a team $X$ if for every assignment $s :
\dom(X) \to M^k$, $s \in X$ iff $s$ satisfies $\varphi$,
\end{guideline}
This guideline replaces the implication in Guideline \ref
{g1} with an equivalence. In the language of first-order logic, nothing
exciting happens; it is only when we introduce new atoms, like dependence
atoms, or new logical operations that interesting things start to happen.

Dependence logic exhibits several appealing basic properties, including
closure under taking subteams and locality (where assignments of variables
not occurring in a formula are irrelevant for determining satisfaction).
However, this new logic is not closed under subteams and does not strictly
adhere to the usual notion of locality. Instead, it possesses a weaker
locality-like property, as demonstrated by Proposition \ref{lem1}.

In the present section, we introduce the semantics of dependence logic in a
slightly non-standard manner. We prove that this alternative semantics is
equivalent to the standard one. In the subsequent section, we modify the
truth conditions somewhat to create a new logic and prove that this new logic
matches the expressive power of existential second-order logic, and
consequently,  with that of dependence logic. Finally, in the last section,
we extend the logic by introducing generalized quantifiers and demonstrate
that the resulting system is conservative over first-order logic and prove
that it can express branching of continuous generlized quantifiers.

\subsection{Dependence logic}

In this section, we provide a brief introduction to dependence logic, where
satisfaction is defined in a non-standard but equivalent manner. For a more
detailed account of dependence logic, please refer to
\cite{Vaananen:2007}.

The syntax of dependence logic extends the syntax of first-order logic by introducing new atomic formulas known as dependence atoms. There is one dependence atom for each arity. We denote the dependence atom expressing that the term $t_{n}$ is uniquely
determined by the values of the terms $t_1,\ldots,t_{n-1}$ by
$\depat(t_1,\ldots,t_n)$.\footnote{The dependence atom is often also denoted by
$\mathord=(t_1,\ldots,t_n)$.} We assume that all formulas of dependence logic
are written in negation normal form, meaning that all negations in formulas appear in front of atomic formulas. Given a vocabulary $\tau$, $\dep[\tau]$ denotes the set
of $\tau$-formulas of dependence logic, i.e., $\dep[\tau]$ is the set described
by 
$$ \varphi ::= \mathrm{At} \mathrel| \lnot \mathrm{At} \mathrel|
\depat(t_1,\ldots,t_n) \mathrel| \varphi \iland \varphi \mathrel| \varphi \ilor
\varphi \mathrel | \exists x \, \varphi \mathrel| \forall x \, \varphi,$$
where $\mathrm{At}$ is an atomic formula in the vocabulary $\tau$.

We denote conjunction and disjunction with $\iland$ and $\ilor$ respectively because we intend to introduce the operators $\land$ and $\lor$ with
a different meaning later in the paper.

The set of free variables of a formula is defined as in first-order logic by
treating the dependence atoms as any other atom. The set of free variables of
a formula $\varphi$ is denoted by $\FV(\varphi)$. We say that a formula is
first-order if it contains no dependence atom.

To define a compositional semantics for dependence logic we employ \emph
{sets of assignments} known as \emph{teams}, rather than single assignments
as in first-order logic. An assignment is a function $s: V \to M$, where $V$
is a finite set of variables and $M$ is the universe under consideration.
Given a universe $M$, a team $X$ over $M$ is a pair of a finite set of
variables, denoted $\dom(X)$, and a subset of the function space $\dom
(X) \to M$. We will use an abuse of notation and refer to the set of
functions as $X$. Thus, we think of a team as a set of assignments along with
a set $\dom(X)$ that specifies the domain of these assignments. Clearly, the
set $\dom(X)$ is determined by $X$ in all cases except when $X$ is empty.

If $V=\emptyset$ there is only one assignment: the empty assignment, denoted
by $\epsilon$. It is important to observe that the team consisting of the empty assignment $\set{\epsilon}$ is different from each empty team; remember, there is one empty team for each finite set of variables. We denote an empty team by $\emptyset_V$, where $V$ is the domain of the team.

Given an assignment $s: V \to M$ and $a_1,\ldots,a_k \in M$ let
$$s[a_1,\ldots,a_k/x_1,\ldots,x_k]: V \cup \set{x_1, \ldots,x_k} \to M$$ 
be the assignment 
$$
s[a_1,\ldots,a_k/x_1,\ldots,x_k]: y \mapsto  \begin{cases}
s(y)  &\text{ if $y \in V \setminus \set{x_1,\ldots,x_k}$, and}\\
a_i  &\text{ if $y=x_i$.}
\end{cases} $$ 

We use the notation $\bar a$ as a shorthand for a finite sequence $a_1,\ldots,
a_k$, but we also treat $\bar a$ as the finite set of the $a_i$'s when
appropriate. The interpretation of the term $t$ in the model $\M$ under the
assignment $s$ is denoted by $t^{\M,s}$. Furthermore, the tuple obtained by
point-wise application of $s$ to the finite sequence $x_1, \ldots, x_k$ is
denoted by $s(\bar{x})$.

For assignments $s: V \to M$ and first-order formulas $\varphi$ such that
$\fv(\varphi) \subseteq V$ we denote the ordinary Tarskian satisfaction by
$\M,s \models \varphi$. If $\FV(\varphi) \subseteq \set{\bar x}$ then
the semantic value or denotation of $\varphi$ is given by
$$\scott{\varphi}^{\M}_{\bar x} = \set{s: \set{\bar x} \to M | \M,s \models
\varphi},$$ a team with domain $\set{\bar x}$.

Just as union and intersection correspond to disjunction and conjunction, there
are operations on denotations that correspond to quantification: 
$$\exists x
X = \set{s : \dom(X) \setminus \set{x} \to M | \exists a \in M : s
[a/x] \in X \text{ or } s \in X}$$ and $$\forall x X = \set{s : \dom(X)
\setminus \set{x} \to M | \forall a \in M : s[a/x] \in X \text{ or } s \in
X}.$$ 
Please note that, for these operations to be properly defined, we need to
add that $$\dom(\exists x X) = \dom(\forall x X) = \dom(X)
\setminus \set{x},$$ which is important in the case when
$\exists x X = \emptyset_V$. Also, observe that $\exists x X = \forall x X = X$
if $x \notin
\dom(X)$.

\begin{lemma} Let $\FV(\varphi) \subseteq \bar y$ and $\varphi$ a first-order
 formula, then \begin{itemize} 	 \item $\exists x \scott{\varphi}^{\M}_
 {\bar y} = \scott{\exists x \varphi}_{\bar y \setminus \set{x}}^
 {\M}$ and \item $\forall x \scott{\varphi}^{\M}_{\bar y} = \scott
 {\forall x \varphi}_{\bar y \setminus \set{x}}^{\M}$. \end{itemize}
\end{lemma}
\begin{proof} Easy to check directly, but note that if $x \notin \bar y$, then
 $\scott{\exists x \varphi}_{\bar y}^{\M} =  \scott{\varphi}_{\bar y}^
 {\M} = \exists x \scott{\varphi}_{\bar y}^{\M}  $.
\end{proof}

Let us now turn to the satisfaction relation for dependence logic. The
following definition, while not the conventional one found in the literature,
is equivalent to it, see Proposition \ref{ordsat}.

\begin{defin}\label{satdef}
The satisfaction relation for dependence logic $\M ,X \models  \varphi$, for $\FV(\varphi) \subseteq \dom(X)$ and $\varphi \in D[\tau]$ is defined as follows. 
\begin{enumerate}
\item $ \M ,X 
	\models  \psi \text{ iff } \forall s \mathord\in X: \M,s \models \psi$,  for first-order atomic or
	negated atomic formulas $\psi$.
\item $\M ,X \models  \depat(t_1,\ldots,t_{n+1}) \text{ iff }
\forall s,s' \mathord\in X \bigwedge_{1\leq i \leq n}t_i^{\M,s}
=t_i^{\M,s'}  \rightarrow  t_{n+1}^{\M,s} =t_{n+1}^{\M,s'}$.
%\item $\M ,X\models \lnot \depat(t_1,\ldots,t_{n+1}) \text{ iff } X =
%\emptyset $
\item $\M ,X \models  \varphi \iland \psi \text{ iff }\exists Y,Z \text{ s.t.  
	} X= Y \cap Z, \text{ and both } \M ,Y \models \varphi \text{ and } \M 
	,Z \models  \psi$.
\item $\M ,X \models  \varphi \ilor \psi \text{ iff } \exists Y,Z  \text{ s.t.  
	} X= Y \cup Z, \text{ and both } \M,Y \models  \varphi \text{ and } \M,Z 
	\models \psi$.
\item\label{ex} $\M ,X \models  \exists x \varphi \text{ iff } \exists Y
\text{ s.t. } x \in \dom(Y),
	\exists x Y = \exists x X \text{ and } \M,Y \models \varphi$.
\item\label{all} $\M ,X \models  \forall x \varphi \text{ iff } \exists Y
\text{ s.t. } x \in \dom(Y),
	\forall x Y = \exists x X \text{ and } \M,Y \models \varphi$.
\end{enumerate}
\end{defin}

Note that $X = Y \cup Z$ and $X=Y \cap Z$ implies that $\dom(X) = \dom(Y) = \dom(Z)$. 

Also, note that in \eqref{ex} and \eqref{all} we use $\exists x X$ instead of
just $X$, which might seem unconventional. However, this choice is deliberate
to handle cases where $x \in \dom(X)$: if $x \notin \dom(X)$ then $\exists x
X = X$. The decision to use $\exists x$ as the arity-reducing operation
applied to $X$ will become clearer in the argument below, following the proof
of Proposition \ref{ordsat}.

The operations $\exists x$ and $\forall x$ are arity-reducing operations.
However, the standard way of dealing with quantifiers in dependence logic
involves arity-increasing operations. Let $X$ be a team, then
\begin{itemize}
	\item $X[M/x]$ is the team $\set{s[a/x] | s \in X, a \in M},$
	\item $X[f/x]$ is $\set{s[f(s)/x] | s \in X}$, where $f: X \to M$ and
	\item $X[F/x]$ is $\set{s[a/x] | s \in X, a \in F(s)}$, where $F: X \to \power(M)$.
\end{itemize}

Define $X \rest\bar x = \set{s \rest \bar x | s \in X}$ and $\dom
(X\rest \bar x) = \dom(X) \cap
\set{\bar x}$; $X(\bar x) = \set{s(\bar x)|s \in X}$, and $\rel(X) = X(\bar
x)$, where $\bar x$ is $\dom(X)$ ordered with increasing indices.

\begin{prop}\label{ordsat}
The satisfaction relation $\M,X \models \varphi$ satisfies:
\begin{enumerate}
	\item If $\M, X \models \varphi$ and $X' \subseteq X$ then $\M,X' \models \varphi$.
  \item $\M ,X \models  \varphi \iland \psi \text{ iff } \M,X \models \varphi$ and $\M,X \models
	\psi$.
  \item $\M ,X \models  \exists x \varphi \text{ iff } \exists f: X \to M \text{ s.t. } \M,X[f/x] \models \varphi$.
  \item $\M ,X \models  \forall x \varphi \text{ iff } \M,X[M/x] \models \varphi$.
\end{enumerate}
\end{prop}
\begin{proof}
	(1) is proved by induction on $\varphi$: All cases are trivial except for the quantifer cases;
	assume $\M, X \models \exists x \varphi$ and that $X' \subseteq X$. 
	Let $Y$ be such that $\exists x Y = \exists x X$ and $\M,Y \models \varphi$. 
	Define 
	$$Y' = \set{ s \in Y | s \rest \bar y \in \exists x X'},$$ where $\bar y = \dom(X) \setminus
	\set{x}$. For $\M,X' \models \exists x \varphi$ we want to prove that $\exists x Y' = \exists x
	X'$, since by the induction hypothesis we have $\M,Y' \models \varphi$.

By definition we get that $\exists x Y'  \subseteq \exists x X'$. For the other inclusion, let $s
\in \exists x X'$ which implies that $s \in \exists x X$ and so $s \in \exists x Y$. Let $s' =
s[a/x] \in Y$ (since $x \in \dom(Y)$). $s' \rest \bar y = s \in \exists x X'$
which proves that $s' \in Y'$ giving us $s \in \exists x Y'$.
And thus $\exists x X' = \exists x Y'$.

For the universal quantifier assume $\M, X \models \forall x \varphi$ and that $X' \subseteq X$. 
	Let $Y$ be such that $\forall x Y = \exists x X$ and $\M,Y \models \varphi$. 
	Define 
	$$Y' = \set{ s \in Y | s \rest \bar y \in \exists x X'},$$ where $\bar y = \dom(X) \setminus
	\set{x}$. To see that $\M,X' \models \forall x \varphi$ we need to prove that $\forall x Y' =
	\exists x	X'$, since by the induction hypothesis we have $\M,Y' \models \varphi$.

By definition we get that $\forall x Y' \subseteq \exists x Y'  \subseteq \exists x X'$. 
For the other inclusion, let $s \in \exists x X'$ which implies that 
$s \in \exists x X$ and so $s \in \forall x Y$. Then $s[a/x] \in Y$ for every $a \in M$ (since $x \in \dom(Y)$). 
Also $s[a/x] \rest \bar y = s \in \exists x X'$
which proves that $s[a/x] \in Y'$ for all $a \in M$; giving us $s \in \forall x Y'$.
And thus $\forall x Y' = \exists x X'$.

(2). Assume $\M,X \models \varphi \iland \psi$ and let $Y \cap Z = X$ be such that $\M,Y \models
\varphi$ and $\M,Z \models \psi$. By the induction hypothesis we get that $\M,(Y \cap X) \models
\varphi$ and $\M,(Z \cap X) \models \psi$ and since $(Y \cap X) \cap (Z \cap X) = Y \cap Z \cap
X = X \cap X = X$ we have that $\M,X \models \varphi \iland \psi$.

For (3) assume first that $\M, X \models \exists x \varphi$ and let $Y$ be such that $\M,Y \models
\varphi$ and $\exists xY = \exists xX$. Define $f': \exists x X \to M$ in such a way that for all
$s \in \exists x X$, $s[f(s)/x] \in Y$. Then let $f:X \to M$ be defined by $f(s) = f'(s\rest \bar
y)$, where $\bar y = \dom(X) \setminus \set{x}$. The team $X[f/x] \subseteq Y$ and by (1) we then
get $\M, X[f/x] \models \varphi$.

For the other direction assume that there is an $f: X \to M$ such that $\M,X[f/x] \models
\varphi$. Let $Y = X[f/x]$, then $\exists x Y = \exists x X$.

(4) Assume that $\M,X \models \forall x \varphi$ and let $Y$ witness that. Then $\forall x Y =
\exists x X$ and thus $(\exists x X) [M/x]\subseteq Y$. However $(\exists x
 X) [M/x] = X[M/x]$. Using (1) we get $\M,X[M/x] \models \varphi$. On the
 other hand if $\M,X[M/x] \models \varphi$, we may chose $Y= X[M/x]$.
\end{proof}

For case (3) to be valid, the natural choice of an arity-reducing
operation in Definition \ref{satdef} appears to be $\exists x$: Let's assume that
another arity-reducing operator was used, denoted by $P(X)$, such that $P(X)
\subseteq \exists x X$. Now consider the following scenario: Suppose $$X =\set{\epsilon}[A/x][a/y] \cup \set{x \mapsto
b, y \mapsto c},$$ where $A \neq \emptyset$ and $a \neq c$. For case (3) to hold we need $$\set{y
\mapsto a} \in P(X)$$. This ensures that $X$ does not satisfy the formula $\exists x (y=c)$. The only reasonable
choice for $P$ seems to be $\exists x$. A similar argument demonstrates that $\exists
x$ is the only reasonable arity-reducing operation to choose in case
\eqref{all} of Definition \ref{satdef}.

Proposition \ref{ordsat} established that the satisfaction relation in this paper is
equivalent to the standard one for dependence logic. Consequently, we retain all the familiar properties associated with dependence logic, including locality:

\begin{prop} $\M,X \models \varphi$ iff $\M,X\rest \FV(\varphi) \models \varphi$. 
\end{prop}

We define $\M \models \sigma$ for a sentence $\sigma$ to hold if $\M
,{\set{\epsilon}} \models \sigma$.

\section{A team logic for generalized quantifiers}

The reader may wonder why we took the extra effort of defining dependence
logic in this slightly roundabout way. The main reason for doing so is to set
up the definitions in such a way that only a small alteration gives us a logic
which is better suited for generalized quantifiers.

This small alternation is based on Guideline \ref{g2} which states that a team satisfies a formula iff the
team \emph{is} the semantic value of the formula:
\addtocounter{guideline}{-1}\begin{guideline}
A formula $\varphi$ is satisfied by a team $X$ if for every assignment $s :
\dom(X) \to M^k$, $s \in X$ iff $s$ satisfies $\varphi$.
\end{guideline}

The idea to formulate this ``maximal'' semantics originates from \cite
{engstrom2012generalized} where a semantics for dependence logic with
non-monotone generalized quantifiers is given using a maximality condition on
the truth condition for the generalized quantifier. However, the approach in that paper takes us beyond the realm of existential second-order logic. Instead, we propose to base the semantics on the guideline above.

In his PhD thesis \cite{nurmi2009dependence} Nurmi presents a similar
semantical system: 1-semantics for the syntax of dependence logic. In
formulating 1-semantics, Nurmi uses the syntax of dependence logic with a
dependence atom that is satisfied by a team iff the team is the graph of a
total function. He proves that a team satisfies a formula in dependence logic
iff there is a superset of the team that satisfies the formula in 1-semantics.
In the present paper, instead of using a dependence atom, we opt for an
\emph{external} disjunction and conjunction and restricts the use of negation
 to atoms; but are otherwise using the same basic idea and guiding principle
 for the semantics as Nurmi. However, we do arrive at slightly different
 conclusions, and it appears that these seemingly small differences give
 rise to quite distinct logics.  

\subsection{mt-logic}

We will define a logic, which we will call \emph{mt-logic} for maximal team logic,
using a variation of Definition \ref{satdef} in which clause (1) is replaced
by a clause that is directly copied from Guideline \ref{g2}. Instead of the
dependence atoms, we will introduce two new connectives: $\land$ and $\lor$, representing the
\emph{external} conjunction and disjunction, distinct from the 
internal operators $\iland$ and $\ilor$.

Thus, the set of formulas of this logic is $$ \varphi ::= \mathrm
{At} \mathrel| \lnot \mathrm{At}
\mathrel| \varphi
\iland \varphi \mathrel| \varphi \ilor \varphi \mathrel|
\varphi \land \varphi \mathrel| \varphi \lor \varphi \mathrel | \exists x \,
\varphi \mathrel| \forall x \, \varphi,$$ where $\mathrm{At}$ is an atomic
formula in the language of some given signature. 

We say that a formula is first-order if there is no occurrence of $\land$ nor
$\lor$.

\begin{defin}\label{satdefteam}
The satisfaction relation for mt-logic $\M ,X \models  \varphi$, where
$\FV(\varphi) \subseteq \dom(X)$, is defined as follows.
\begin{enumerate}
\item $ \M ,X 
	\models  \psi \text{ iff } \forall s: \dom(X) \to M  (s \in X \text{ iff }
	\M,s \models \psi)$,  for literals
	$\psi$
%\item $\M ,X \models  \top(\bar x) \text{ iff } \exists \bar x X = \set
% {\epsilon}[M^k/\dom(X) \setminus \set{\bar x}]$
\item $\M ,X \models  \varphi \iland \psi \text{ iff }\exists Y,Z \text{ s.t.  
	} X= Y \cap Z; \M ,Y \models \varphi \text{ and } \M 
	,Z \models  \psi$
\item $\M ,X \models  \varphi \ilor \psi \text{ iff } \exists Y,Z  \text{ s.t.  
	} X= Y \cup Z; \M,Y \models  \varphi \text{ and } \M,Z 
	\models \psi$
\item $\M ,X \models  \varphi \land \psi \text{ iff } \M ,X \models \varphi \text{ and } \M 
	,X \models  \psi$
\item $\M ,X \models  \varphi \lor \psi \text{ iff } \M,X \models  \varphi \text{ or } \M,X 
	\models \psi$
\item\label{exteam} $\M ,X \models  \exists x \varphi \text{ iff } \exists Y
\text{ s.t. } x \in \dom(Y),
	\exists x Y = \exists x X \text{ and } \M,Y \models \varphi$
\item\label{allteam} $\M ,X \models  \forall x \varphi \text{ iff } \exists Y
\text{ s.t. } x \in \dom(Y),
	\forall x Y = \exists x X \text{ and } \M,Y \models \varphi$
\end{enumerate}
\end{defin}
If there is need, we will denote this satisfaction relation by
$\models_\text{mt}$ and the relation for dependence logic by $\models_D$. We
will also often skip to mention the model $\M$ if that is understood from the
context.

Instead of using arity-reducing operations on teams in the quantifier clauses
we may use arity-increasing operators as in Proposition \ref{ordsat}; however,
the definition for the universal quantifier then becomes more involved:

\begin{prop} For $\varphi$ in $T[\tau]$ we have:
	\begin{itemize}
		\item $\M,X \mmodels \exists x \varphi$ iff there exists a function $F: \exists xX \to
			\power(M) \setminus \set{\emptyset}$ such that $(\exists xX)[F/x] \mmodels \varphi$.
		\item $\M,X \mmodels \forall x \varphi$ iff there exists a function $F: (\exists
			xX)^c \to
			\power(M) \setminus \set{M}$ such that $(\exists xX)[M/x] \cup (\exists x X)^c[F/x] \mmodels \varphi$.
	\end{itemize}
\end{prop}
\begin{proof}
Directly from the definitions.
\end{proof}

We note that satisfaction of mt-logic can be expressed in existential second
order logic, ESO, i.e., using $\Sigma_1^1$-formulas:

\begin{prop}
For every formula in mt-logic $\varphi$ in the signature $\tau$ with $n$ free
variables there is a $\Sigma_1^1$ formula $\Theta$ in the language of $\tau
\cup \set{R}$, where $R$ is $n$-ary, such that for all $\M$ and $X$: $$\M ,X
\mmodels \varphi \text{ iff } (\M,\rel(X)) \models \Theta.$$
\end{prop}

The proof follows the proof of the similar result regarding dependence logic
in \cite{Vaananen:2007}. The details are left to the reader.

As mentioned above the basic idea behind this logic is that for first-order
formulas $\varphi$, $X \models \varphi$ iff $X$ \emph{is} the semantic value
of $\varphi$, $\scott{\varphi}^{\M}$. This is indeed true for a large class of
formulas, but not true in other cases: Assume $\dom(X)= \set{x}$ then $\M,X
\models \exists x (x=x)$ iff there is $Y$ s.t.\ $x \in \dom(Y)$, $\exists x Y
= \exists x X$ and $\M,Y \models x=x$, i.e., iff $\exists x X =
\set{\epsilon}$. Thus any non-empty $X$ satisfies the sentence.

For this reason we will single out a subclass of formulas that we call
\emph{untangled}: A formula  is untangled if no quantifier $Qx$ appears in the
 scope of another quantifier $Q'x$ and no variable is both free and
 bound.\footnote{Here $Q$ and $Q'$ are either $\forall$ or $\exists$.}

The next proposition shows that untangled first-order formulas are
well-behaved. Here $\BV(\varphi)$ is used to denote the set of bound variables
of $\varphi$.

\begin{prop}\label{lem1} Assume $\varphi$ is an untangled first-order formula in mt-logic and $X$ a
team such that $\dom(X)\cap \BV(\varphi) = \emptyset$. Then
			$$X \mmodels \varphi \text{ iff } X = \scott{\varphi}^{\M}_{\dom(X)}.$$
\end{prop}
\begin{proof}
This is proved by a direct induction over formulas. The base case is follows directly
from the definitions and the cases for $\iland$ and $\ilor$ follows easily.
Let us do the $\exists$ case: If $\varphi$ is $\exists x \psi$, then $x \notin
\BV(\psi)$ and so by the induction hypothesis $Y \models \psi$ iff
$Y=\scott{\psi}^{\M}_{\dom(Y)}$%whenever $\dom(Y)=\dom(X) \cup \{x\}$.
Since $x \notin \dom(X)$ we have that $X \models \exists x \psi$ iff there is
$Y$ such that $\exists xY = X$ and $Y=\scott{\psi}^{\M}_{\dom(X) \cup \{x\}}$.
This is true iff 
\begin{equation*}
X= \exists x \scott{\psi}^{\M}_{\dom(X)\cup\{x\}} =
\scott{\exists x\psi}^{\M}_{\dom(X)}.\qedhere
\end{equation*}
\end{proof}

Even though mt-logic does not exhibit full locality, it does have a weaker
locality property that is outlined in the next proposition. In fact, in the
upcoming section, we will use the absence of complete locality to interpret
the independence atom.

\begin{prop}\label{lem2}
Assume $\varphi$ is a formula in mt-logic
% and $X$ a team such that $\dom(X)\cap \BV(\varphi) = \emptyset$.  If 
and $x \in \dom(X)$ does not occur in $\varphi$, then $$X \mmodels \varphi 
		 \text{ iff } \exists x X \models \varphi  
		 \text{ and } \forall x X = \exists x X.$$
Also, if  $\bar w = \dom(X)\setminus (\FV(\varphi) \cup \BV(\varphi))$ 
then $$X \mmodels \forall \bar w
			\varphi \text{ iff } \exists \bar w X \mmodels \varphi.$$ 
\end{prop}
\begin{proof}
The first statement is proved by induction over formulas. The base case is taken care of 
by Proposition \ref{lem1} since any formula without quantifiers is untangled. The
inductive steps are easy for $\land$ and $\lor$. For the other cases:

Assume $X \models \varphi \ilor \psi$. Then there are $Y
\cup Z=X$ such that $Y \models \varphi$ and $Z \models \psi$, and so by the
induction hypothesis $\exists xY = \forall xY$, $\exists
xY \models \varphi$, $\exists xZ=\forall xZ$, and $\exists xZ\models \psi$.
 Now, $$ \exists x X = \exists x(Y \cup Z) = \exists xY \cup \exists xZ
 =\forall xY \cup \forall xZ \subseteq \forall x (Y \cup Z)=\forall x X,$$
 and thus, since $\forall x X \subseteq \exists x X$ holds in general,
 $\exists x X = \forall x X$. Since $\exists x X = \exists x Y \cup \exists
 xZ$, $\exists xY \models \varphi$ and $\exists xZ\models \psi$ we also have
 $\exists xX \models \varphi \ilor \psi$.

For the other direction assume $\exists xX \models \varphi \ilor \psi$ and
$\exists x X = \forall x X$. Then there are $Y\cup Z = \exists xX$ such that
$Y \models \varphi$ and $Z \models \psi$. Let $Y'=Y[M/x]$ and $Z'=Z
[M/x]$, then $\exists xY'=\forall xY'=Y$ and so $Y' \models \varphi$ by the
induction hypothesis. Thus $Y' \cup Z' \models \varphi \ilor \psi$ and
$Y' \cup Z' = X$.
			
The case of $\varphi \iland \psi$ is similar. 
			
Assume $X \models \exists y \varphi$. Then there is $Y\models \varphi$ such
that $\exists y Y = X$. By the induction hypothesis we know that $\exists
xY \models \varphi$ and $\exists xY = \forall xY$. Thus, \begin
{equation}\label{eq:1} \exists x X = \exists x \exists y Y = \exists y\exists
xY=\exists y \forall xY \subseteq \forall x\exists y Y = \forall x X. \end
{equation} But clearly $\forall x X \subseteq \exists x X$ and so $\forall x
X = \exists x X$. Also, $\exists x Y \models \varphi$ and $\exists y\exists
xY = \exists x X$ by \eqref{eq:1} and thus, $\exists x X \models \exists
y \varphi$.

For the other direction assume $\exists xX \models \exists y \varphi$ and
$\exists x X = \forall x X$. Then there is $Y \models \varphi$ such that
$\exists y Y = \exists x X$. Let $Y'=Y[M/x]$, then $\exists xY' = \forall x
Y'=Y$ and so, by the induction hypothesis $Y' \models \varphi$. Thus,
$\exists y Y' \models \exists y \varphi$ and $$\exists yY'= \exists y(Y
[M/x])=(\exists yY)[M/x] = (\exists x X)[M/x] = X.$$

Assume $X \models \forall y \varphi$. Then there is $Y\models \varphi$ such
that $\forall y Y = X$. By the induction hypothesis we know that $\exists
xY \models \varphi$ and $\exists xY = \forall xY$. Thus, \begin
{equation}\label{eq:2} \exists x X = \exists x \forall y Y \subseteq \forall
y\exists x Y=\forall y \forall xY = \forall x\forall y Y = \forall x X. \end
{equation} But clearly $\forall x X \subseteq \exists x X$ and so $\forall x
X = \exists x X$. Also, $\exists x Y \models \varphi$ and $\forall y\exists
xY = \exists x X$ by \eqref{eq:2} and thus, $\exists x X \models \forall
y \varphi$.

For the other direction assume $\exists xX \models \forall y \varphi$ and
$\exists x X = \forall x X$. Then there is $Y \models \varphi$ such that
$\forall y Y = \exists x X$. Let $Y'=Y[M/x]$, then $\exists xY' = \forall x
Y'=Y$ and so, by the induction hypothesis $Y' \models \varphi$. Thus,
$\forall y Y' \models \forall y \varphi$ and $$\forall yY'= \forall y(Y
[M/x])=(\forall yY)[M/x] = (\exists x X)[M/x] = X.$$
		
For the also-part we note that $X \models \forall \bar w \varphi$ iff there is
$Y$ such that $\forall \bar w Y =\exists \bar w X$ and $Y \models \varphi$. By
the previous part of the proof, $Y \models \varphi$ iff $\forall \bar w Y =
\exists \bar wY$ and $\exists \bar wY \models \varphi$. Thus, the left to
right implication follows directly. For the other, assume $\exists \bar w X
\models \varphi$ and let $Y=(\exists \bar w X)[M/\bar w]$. Then $Y \models
\varphi$ and $\forall \bar w Y= \exists \bar w Y=\exists \bar w X$. Thus, $X
\models \varphi$.
\end{proof}

Observe also that for all $\varphi$ there is a team $X$ such that $\M,X
\mmodels \varphi$, this follows by an easy inductive argument.

\subsection{Relationship with dependence and independence logic}

Next, we investigate the relationship between mt-logic on the one hand and
dependence and independence logic on the other. We will work our way towards
interpretations of dependence and independence logic in mt-logic. Let us
first see how to deal with atoms of dependence logic that are closed
downwards in mt-logic:

\begin{prop}
	If $\psi$ is a first-order formula then
\begin{equation}\label{eq_literal}
X \models \psi \iland \exists \bar w (w_0=w_0 \lor w_0 \neq w_0)\text{ iff }
X \subseteq \scott{\psi}_{\bar w},
%\forall s \in X: s \models \psi,
\end{equation}
where $\bar w$ is $\dom(X)$ and $w_0 \in \bar w$.
\end{prop}
\begin{proof}
This is easily seen by first observing that any team $Y$ with domain $\bar w$
satisfies $\exists
\bar w (w_0=w_0\lor w_0 \neq w_0)$: If $Y$ is empty then $Z=\emptyset$
satisfies $w_0=w_0\lor w_0 \neq w_0$ and $\exists \bar w Z = \exists \bar w
Y= \emptyset$. In the other hand, if $Y$ is non-empty then $Z=\{\epsilon\}[M/\bar w]$ satisfies
$w_0=w_0 \lor w_0 \neq w_0$ and $\exists \bar w Y = \exists \bar w Z = \{\epsilon\}$.

Now, the left hand side of \eqref{eq_literal} holds iff there is $Y$ such that
$X =\scott{\psi}_{\bar w} \cap Y$, which is equivalent to $X \subseteq
\scott{\psi}_{\bar w}$.
\end{proof}

Thus, $X \models_\text{mt} \psi \iland \top_{\bar w}$ iff $X \models_\dep \psi$,
where $\top_{\bar w}$ denotes the sentence $\exists \bar w (w_0=w_0 \lor w_0
\neq w_0)$; and so we have an interpretation of the literals of dependence
logic in mt-logic.

Observe that, in general when $\bar w \subseteq \dom(X)$ we have $X \models
\truth_{\bar w}$ iff  $X= \emptyset$ or $\exists
\bar w X$ is the full team with domain $\dom(X) \setminus \bar w$.

Instead of directly interpreting the dependence atom we turn to the
\emph{independence atom} of \cite{gradel2012dependence}\footnote
 {The independence atom was introduced in \cite{gradel2012dependence} and is
 the embedded version of the multivalued dependence relation used in database
 theory, see
\cite{engstrom2012generalized}.}: $X \models_\dep \bar y \indep_{\bar x} \bar z$
iff $$ \forall s,s'\in X \exists s_0 \in X \bigl( s(\bar x)=s'(\bar x)
\rightarrow s_0(\bar x,\bar z)=s(\bar x,\bar z) \land s_0(\bar x,\bar y) =
s'(\bar x,\bar y)\bigr)$$ We show that this is expressible in mt-logic. 

%We will use the following short-hand:$$\C_{\bar y}X = (\exists \bar y X)[M/\bar y]$$

\begin{prop}\label{prop_indep} 
If $\bar x$, $\bar y$, and $\bar z$ are pair-wise disjoint then
$$X \models_\dep \bar y \indep_{\bar x} \bar z \text{ iff } X\mmodels \forall \bar
w ( \truth_{\bar x,\bar y} \iland \truth_{\bar x,\bar z}),$$ where $\bar w$ is
$\dom(X) \setminus \set{\bar x,\bar y,\bar x}$. 
\end{prop}

\begin{proof}
Observe first that, by Proposition \ref{lem2}: $$X \mmodels \forall \bar w(
\truth_{\bar x,\bar y} \iland \truth_{\bar x,\bar z})  \text{ iff } \exists \bar w X \mmodels 
\truth_{\bar x,\bar y} \iland \truth_{\bar x,\bar z}.$$ 
Therefore we may, without loss of
generality, assume that $\dom(X) = \set{\bar x,\bar y,\bar z}$.

By using the fact that $\bar y$ and $\bar z$ are disjoint we get that $X
\mmodels \truth_{\bar x,\bar y} \iland \truth_{\bar x,\bar z}$ iff
\begin{equation}\label{eq1}
	X = X_{\bar z} \cap X_{\bar y},
\end{equation}
where $X_{\bar z} = (\exists \bar z X)[M/\bar z]$ and similar for $X_{\bar y}$.

Now, assume that $X \models \bar y \indep_{\bar x} \bar z $ and prove  \eqref{eq1}. The left-to-right
inclusion is trivial so let us assume that $s_1 \in X_{\bar z} \cap X_{\bar y}$, 
i.e., that there are $s$ and $s'$ such that $s(\bar x,\bar y) = s_1(\bar x,\bar
y)$ and $s'(\bar x,\bar z) = s_1(\bar x,\bar z)$. This tells us that there is $s_0
\in X$ such that $s_0=s_1$, and thus that $s_1 \in X$.

On the other hand, assume \eqref{eq1} and that $s,s' \in X$ such that $s
(\bar x)=s'(\bar x)$. Define $s_0$ to be such that $s_0(\bar x,\bar z)=s
(\bar x,\bar z)$ and $s_0(\bar x,\bar y) = s'(\bar x,\bar y)$. From these two
equations it follows that $s_0 \in X_{\bar z} $ and  $s_0 \in X_{\bar y} $
and thus, from \eqref{eq1}, we get that $s_0 \in X$. Thus,
$X \models_\dep \bar y \indep_{\bar x} \bar z$.
\end{proof}

In fact, disjoint independence atoms can express any independence atom by
existentially quantifying in new variables. As can be easily checked, $D
(\bar x,y)$ is equivalent to $y \indep_{\bar x} y$ which in Dependence logic
is equivalent to $\exists z (y \indep_{\bar x} z \land y=z)$. Thus, we may
express dependence atoms in mt-logic as follows.

\begin{prop}\label{prop_dep_atom}
$X \models_\dep \depat(\bar x,y)$ iff $$X \models_\mt \exists z \bigl(\forall \bar w (
\truth_{\bar x, y} \iland \truth_{\bar x, z})  \land (y=z \iland \truth_
 {\bar x,\bar w})\bigr),$$where $z$ is not in $\bar x,y$ and $\bar w$ is
 $\dom(X)\setminus \set{\bar x,y,z}$.
\end{prop}
\begin{proof} Assume that $X \models_\dep D(\bar x,y)$ and define $Y=X[f/z]$ where
 $f(s)=s(y)$. According to Proposition \ref{prop_indep} $Y$ satisfies the
 first conjunct of the formula and according to \eqref{eq_literal} it
 satisfies the second conjunct.

On the other hand, assume there is $Y$ such that $X= \exists z Y$ and $Y$
satisfies the conjunction. $Y$ satisfies the second conjunct iff $s(z)=s
(y)$ for all $s \in Y$. Satisfying the first conjunct implies that for all
$s,s' \in Y$ if $s(\bar x) = s'(\bar x)$ then there is $s_0 \in Y$ such that
$s_0(\bar x,y)=s(\bar x,y)$ and $s_0(z) = s'(z)$, i.e., $s(y) = s_0(y)=s_0
(z)=s'(z) = s'(y)$. In other words, $X \models_\dep D(\bar x,y)$.
\end{proof}

% \begin{lemma} Assume $\M$ has at least two elements, then for any team $X$:
% 	$\M,X \mmodels \exists v (D(\bar x,v)\land D(\bar x,v))$, where $\dom(X) = \set{\bar x}$, and $v
% 	\notin \dom(X)$. 
% \end{lemma}
% \begin{proof} 
% 	Let $b\neq c \in \M$.
% 	Define $f,g: M^k \to M$, where $\bar x = x_1,\ldots,x_k$ such that $f(\bar a) = b$ for all $\bar
% 	a \in M^k$ and $g(\bar a) = b$ whenever $\set{\bar x \mapsto \bar a} \in X$ and $g(\bar a)= c$
% 	otherwise. Let $Y = \set{s: \set{\bar x,v} \to M| s(v) = f(s(\bar x))}$ and $Z =  \set{s:
% 	\set{\bar x,v} \to M| s(v) = g(s(\bar x))}$. Then $Y \cap Z = X[f'/v]$, where $f' = f \rest X$.

% 	Now $\exists v (X[f'/v]) = X$ and so $X \mmodels \exists v (D(\bar x,v)\land D(\bar x,v))$.
% \end{proof}

% Let $\truth_{\bar x;v}$ be the formula $\exists v (D(\bar x,v)\land D(\bar x,v))$. We will in general assume that $v \notin \bar x$.
% Observe that when $|\M| = 1$, $\emptyset_{\bar x} \not\mmodels \truth_{\bar x;v}$. 

We are now able to define an \emph{almost} compositional translation ${}^+:
\varphi \mapsto \varphi^+$ of dependence logic into our logic in such a way
 that $$ \M,X \models_\dep \varphi \text{ iff } \M,X \mmodels \varphi^+,$$
 for all models $\M$ and teams $X$ with $\dom(X) = \FV(\varphi)$. This is
 done by replacing atomic formulas $\psi$ by $\psi \iland \truth_{\bar x}$
 for some suitable choice of variables $\bar x$ and using Proposition \ref
 {prop_dep_atom}.

\begin{defin} Let $f(\bar w,\varphi)$ be defined inductively on the set of
 dependence logic formulas: 
\begin{itemize} 
\item $f(\bar w,D(\bar x,y)) = \exists z (\forall \bar w' (\truth_{\bar x, y} \iland \truth_{\bar x, z})  \land (y=z \iland \truth_{\bar x,\bar w'})$, where $\bar w' = \bar w\setminus \set{\bar x,y,z}$ and $z$ is not in $\bar w$.
\item $f(\bar w,\varphi) = \varphi \iland \truth_{\bar w}$ if $\psi$ is a literal,
\item $f(\bar w,\varphi \land \psi) = f(\bar w,\varphi) \iland f(\bar w,\varphi)$, 
\item $f(\bar w,\varphi \lor \psi) = f(\bar w,\varphi) \ilor f(\bar w,\varphi)$, and
\item $f(\bar w,\exists y \varphi) = \exists y f(\bar w,y,\varphi)$, and 
\item $f(\bar w,\forall y \varphi) = \forall y f(\bar w,y,\varphi)$.
\end{itemize}
Let $\varphi^+$ be the formula $f(\FV(\varphi),\varphi)$.
\end{defin}

The translation $\varphi \mapsto \varphi^+$ is not compositional since for
example $(x=x \land y=y)^+$ is not $(x=x)^+ \iland (y=y)^+$, and these two
formulas are not equivalent. 

\begin{prop}
	For every team $X$ and dependence logic formula $\varphi$ such that
$\dom(X)=\FV(\varphi)$: $$ X \models_\dep \varphi \text{ iff } X \mmodels
\varphi^+.$$
\end{prop}
\begin{proof}
Observe first that the role of $\bar w$ in $f(\bar w,\varphi)$ is to ``keep
track'' of $\dom(X)$; this has to be cared for in the induction step and thus
we will instead prove the slightly more involved statement that if
$\FV(\varphi) \subseteq \dom(X)$ then $$ X \models_\dep \varphi \text{ iff }
X \mmodels f(\dom(X),\varphi).$$ The two base cases are handled by
\eqref{eq_literal} and Proposition \ref{prop_dep_atom}. For the inductive
steps observe that the satisfaction clauses for $\ilor$ and $\iland$
correspond exactly to the satisfaction clauses in dependence logic for $\land$
and $\lor$. Similarly  with the two quantifier cases.
\end{proof}

%\begin{prop}
%For every formula $\varphi$ and team $X$, if $X \mmodels \varphi$ then $X \models
% 			\varphi$.
%\end{prop}

We may, in a similar fashion, give a translation $g$ of independence logic
into team logic in such a way that 
$$ X \models_\dep \varphi \text{ iff } X \mmodels g(\varphi,\dom(X)).$$ 
Note that this translation need to take care
of the non-disjoint independence atom by translating it into disjoint
independence atom: In independence logic we have that $$ \bar y \indep_
{\bar x} \bar z \text{ is equivalent to } \exists \bar v,\bar w (\bar v=\bar
y \land \bar w = \bar z \land \bar v \indep_{\bar x} \bar w).$$

Galliano in \cite{galliani2012inclusion} proved that every ESO-property can be
expressed by an independence formula and thus team logic has the same
expressive power:
\begin{thm} The expressive power of team logic is that of existential
 second-order logic, for both formulas and sentences.
\end{thm}

\section{Generalized quantifiers}

According to Mostowski \cite{Mostowski:1957} and Lindström \cite
{Lindstrom:1966} a \emph{generalized quantifier} is a class (in most cases
a \emph{proper class}) of structures in some finite relational signature
closed under taking isomorphic images. For example $$\text{\tt most} = \set{
(M,A,B) : |A \cap B| \geq |A \setminus B|}$$ is a generalized quantifier.The
truth condition is defined so that $$
\M,s \models \mathtt{most}\, x,y \ (\varphi(x), \psi(y)) \ \text{ iff } \
 (\M,\scott{\varphi}^s_x,\scott{\psi}^s_y) \in \mathtt{most},$$
where $$\scott{\varphi}^s_x = \set{a \in \M | \M,s[a/x] \models \varphi}.$$

Thus, 
$$
\M,s \models \mathtt{most}\, x,y \ (\varphi(x), \psi(y))
 \text{ iff } \ |\scott{\varphi}^s_x \cap \scott{\psi}^s_y| \geq |\scott{\varphi}^s_x \setminus \scott{\psi}^s_y|,
$$
which coincides with the intuitive truth condition that most $\varphi$'s are $\psi$'s.

Given a generalized quantifier $Q$ and a domain $M$, let the \emph
{local} quantifer $Q_M$ be defined as $$Q_M = \set
{\langle A_0,A_1,\ldots,A_k \rangle | (M,A_0,A_1,\ldots,A_k) \in Q}.$$

% Observe that local quantifiers are just sets of relations over the domain $M$, they are not generalized quantifiers in the strict sense. Generalized quantifiers in the strict sense we sometimes call \emph{global} when need is to distinguish them from \emph{local} quantifiers. 

We say that a generalized quantifier is of \emph{type} $\langle
n_1,\ldots,n_k\rangle$ if it is a class of structures in the relational
signature $\set{R_1,\ldots,R_k}$ where $R_i$ is of arity $n_i$. 

Given a generalized quantifier $Q$ of type $\langle n \rangle$ we may
extend mt-logic with it in such a way that (2) in
Lemma \ref{lem1} is holds for all $\text{FO(Q)}$ formulas. This is done
by the following definition:
\begin{defin} \label{gq}
	Let $Q$ be of type $\langle n\rangle$ then 
	$\M, X \mmodels Q\bar x \varphi$
	iff
  there is $Y$ such that $\bar x \in \dom(Y)$, $\M,Y\mmodels \varphi$ and  
	$\exists \bar x X = Q\bar x Y$,
	where $$Q\bar x Y = \set{s :\dom(Y) \setminus \set{\bar x} \to M | Y_s(\bar x) \in Q_M}.$$
\end{defin} Remember that $Y_s = \set{s': \dom(Y) \setminus \dom(s) \to M |
 s\cup s' \in Y}$. Observe that the clauses for $\forall$ and $\exists$ in
 Definition \ref{satdef} are special cases of the above definition.

Similar to the cases of $\exists$ and $\forall$ there is an alternative truth condition for generalized quantifiers $Q$ using arity increasing operations:

\begin{lemma} $X \mmodels Q \bar x \varphi$ iff there exists a function $$F:
 \set{\epsilon}[M/\dom(X) \setminus \set{\bar x}] \to \power(M^k)$$ such that $F
 (s) \in Q_M$ iff $s \in \exists \bar x X$ and $(\exists \bar x X)
 [F/\bar x] \mmodels \varphi$.
\end{lemma}
\begin{proof}
Directly from the definitions.
\end{proof}

\begin{prop}\label{prop:cons} For every untangled $\varphi$  formula of $FO(Q)$ and every team
 $X$ such that $\dom(X) \cap \BV(\varphi) = \emptyset$:
 $$\M,X \mmodels \varphi\text{ iff } X = \scott{\varphi}^{\M}_{\dom(X)}.$$
\end{prop}
\begin{proof} 
Induction as in the case with FO-formulas in Lemma \ref
 {lem1}. The only new case is when $\varphi$ is $Q x \psi$ and $x \notin \BV
 (\psi)$. By the induction hypothesis $Y \models \psi$ iff $Y=\scott{\psi}^
 {\M}_{\dom(Y)}$. Since $x \notin \dom(X)$ we have that $X \models Q
 x \psi$ iff there is $Y$ such that $Q x Y = X$ and $Y=\scott{\psi}^
 {\M}_{\dom(X) \cup \{x\}}$. This is true iff 
\begin{equation*}
X= Q x \scott{\psi}^{\M}_{\dom(X)\cup\{x\}} =
\scott{Q x\psi}^{\M}_{\dom(X)}.\qedhere
\end{equation*}
\end{proof}

The next lemma shows that our definition is true to the truth conditions of
monotone increasing generalized quantifiers in Dependence logic introduced
in \cite{engstrom2012generalized}. A generalized quantifier $Q$ of type
$\langle n \rangle$ is \emph{monotone increasing} if $R \in Q_M$ and
$R \subseteq S$ implies $S \in Q_M$. We remind the reader of the truth condition when $Q$ is
monotone increasing:  $X \models_\dep Q\bar x \varphi$ iff there exists a
function $F: \exists \bar x X \to Q_M$ such that $(\exists \bar x X)
[F/\bar x] \models_\dep \varphi$.

\begin{lemma} 
If $Q$ is monotone increasing of type $\langle n\rangle$, then
$X \models_\dep Q\bar x \varphi$ iff there exists 
$Y$ with $\bar x \in \dom(Y)$, $Q\bar x Y = \exists \bar x X$ and $Y \models_\dep \varphi$.
\end{lemma}
\begin{proof} First observe that we may assume $\bar x \notin \dom(X)$ to
 simply notation. Then assume $X \models_\dep Q\bar x \varphi$ and let $F:
 X \to Q_M$ be such that $X[F/\bar x] \models_\dep \varphi$. We may now chose
 $Y=X[F/\bar x]$ satisfying $Y \models \varphi$ and $Q\bar x Y = X$. 

For the other direction, assume that there is a $Y \models_\dep \varphi$ such
that  $Q \bar x Y = X$. Define $F: X \to \power(M^k)$ by $F(s) = \set{\bar a
| s[\bar a /\bar x] \in Y}$. Clearly $X[F/\bar x] \subseteq Y$ proving that
$X[F/\bar x] \models_\dep \varphi$. Also, if $s \in X = Q\bar x Y$, then $F
(s) \in Q_M$.
\end{proof}
%$F: \exists \bar x X \to Q_M$ such that $(\exists \bar x X) [F/\bar x] \models_\mt \varphi$.
%$\bar x$ some variables and $X$ a team, then the following are equivalent:
% \begin{enumerate}
% 	\item There is $Y$ such that $\bar x \in \dom(Y)$, $\M,Y\mmodels \varphi$ and  
% 	$\exists \bar x X = Q\bar x Y$.
% 	\item There is a function $F: X \to Q_M$ such that $(\exists \bar x X)
%  [F/\bar x] =$.
% \end{enumerate}
% \end{lemma}
%$X \models_\dep Q\bar x \varphi$ iff there exists a function 
%$F: \exists \bar x X \to Q_M$ such that $(\exists \bar x X) [F/\bar x] \models_\mt \varphi$.
%$Y$ with $\bar
%x \in \dom(Y)$, $Q\bar x Y = \exists \bar x X$ and $Y \models \varphi$.

In \cite{engstrom2013characterizing} the strength of $D(Q)$ is characterized
for monotone increasing quantifiers $Q$ as the strength of $\text{ESO}
(Q)$. Since we may interpret $D(Q)$ in mt-logic extended with $Q$ we see that
this logic is at least as strong as $\text{ESO}(Q)$. By a standard argument
we can ``interpret'' mt-logic with $Q$ in $\text{ESO}(Q)$ proving that
mt-logic with $Q$ has the same strength as $\ESO(Q)$ for monotone
increasing quantifiers $Q$:

\begin{thm} 
Mt-logic extended with a monotone increasing $Q$ has the
same strength as $D(Q)$, and thus as $\ESO(Q)$.
\end{thm}

As we say in Proposition \ref{prop:cons} we also have a conservativity result,
at least for a large class of $\FO(Q)$-formulas. Before discussing branching
quantifiers we also note that our definition of a generalized quantifier
respects iterated quantifiers sense below. The iteration of two type $\langle
1 \rangle$ quantifiers $Q$ and $Q'$ is defined by $$ (Q \cdot Q')_M = \set
{R \subseteq M^2 | \set{ a | R_a \in Q'_M} \in Q_M },$$ where $R_a = \set
{b | \langle a,b\rangle \in R}$.

\begin{thm}
For any two generalized quantifiers $Q$ and $Q'$ of type  $\langle 1 \rangle$:
$$ \M,X \mmodels (Q \cdot Q') xy \, \phi \, \text{ iff }\, 
\M,X \mmodels Q x \, Q' x \, \phi. $$
\end{thm}
\begin{proof} It is clearly enough to see that $Qx (Q'y Y) = (Q\cdot Q')xy Y$
 for all teams $Y$ including $x$ and $y$ in its domain. But this follows
 immediate from the definitions. 
\end{proof}

\section{Branching quantifiers}

We conclude this paper by discussing the branching of generalized quantifiers.
One of the motivations behind Hintikka and Sandu’s introduction of
IF-logic \cite{Hintikka:1989} was to express the branching behavior of
quantifiers. In \cite{engstrom2012generalized}, it was demonstrated that
there exists a natural way to express the branching of two monotone
increasing quantifiers within what is now known as Independence logic. In
this section, we establish that mt-logic is sufficiently powerful to capture
the branching behavior of a broader class of quantifiers, specifically the
continuous ones

Barwise (see \cite{barwise1979branching}), among others, argues that for
increasingly monotone quantifiers $Q_1$ and $Q_2$ of type $\langle 1 \rangle$ the
branching of $Q_1$ and $Q_2$ $${Q_1x}{Q_2y}A(x,y)$$ should be interpreted as
$$\Br(Q_1,Q_2)xy \ A(x,y),$$ where $\Br(Q_1,Q_2)$ is the type $\langle
2 \rangle$ quantifier $$\set{(M,R) | \exists A \in Q_1, B \in Q_2, A \times
B \subseteq R}.$$

Westerståhl, in \cite{westerstaahl1987branching}, suggests a definition of the
branching of a larger class of generalized quantifiers; 
\emph{continuous} generalized quantifiers. These quantifiers have also been
 studied under different names: convex quantifiers \cite
 {gierasimczuk2023convexity} and connected quantifiers \cite
 {chemla2019connecting}. For simplicity we only give the definitions here for
 quantifiers of type $\langle n \rangle$:

\begin{defin}
	\begin{itemize}
		\item A quantifier $Q$ is \emph{continuous} if for every $M$ and every $R_1 \subseteq R_2 \subseteq
		R_3$ such that $R_1,R_3 \in Q_M$ we have $R_2 \in Q_M$.
		\item The branching $\text{Br}(Q_1,Q_2)$ of two continuous quantifiers $Q_1$ and $Q_2$ is
			defined by  $R \in \text{Br}(Q_1,Q_2)_M$ iff there exists $S_1,S_1' \in {Q_1}_M$ and
			$S_2,S_2' \in {Q_2}_M$ such
			that $S_1 \times S_2 \subseteq R \subseteq S_1'\times S_2'$.
	\end{itemize}
\end{defin}

In mt-logic we are able to express the branching of two continuous quantifiers as the following propositions shows.

\begin{prop}
Let $Q_1$ and $Q_2$ be continuous quantifiers of type  $\langle 1 \rangle$ and $\varphi$ a first-order formula, then
$\M \mmodels \text{Br}(Q_1,Q_2) xy \, \phi \,$ iff  
\begin{multline}\label{branchingformula}
\M \mmodels  Q_1x \, Q_2y \, Q_1z \, Q_2w (x \indep y \land xy \indep z \land xyz\indep w \land\\  (\varphi
(x,y) \iland \top_{xyzw})  \land (\varphi(z,w) \ilor \top_{xyzw})). 
\end{multline}
\end{prop}
\begin{proof} 

First note that $X \mmodels Q_1x \, Q_2y \, Q_1z\, Q_2w\, (x \indep y \land
xy \indep z \land xyz\indep w )$ iff $X = \{\epsilon\}[A/x][B/y][C/z]
[D/w]$ where $A,C \in Q_{1M}$ and $B,D \in Q_{2M}$. 

Secondly, such an $X$ satisfies $\varphi (x,y) \iland \top_{xyzw}$ iff
$A \times B \subseteq R$ where $R$ is the relation corresponding to the team
$\scott{\varphi}_{x,y}$. Similarly $X$ satisfies $\varphi (z,w) \iland \top_
{xyzw}$ iff $R \subseteq C \times D$. 

Thus \eqref{branchingformula} holds iff there are $A,C \in Q_{1M}$ and
$B,D \in Q_{2M}$ such that $A \times C \subseteq R \subseteq B,D$, i.e., iff
$R \in \text{Br}(Q_1,Q_2)$. Which is equivalant to $\M \mmodels \text{Br}
(Q_1,Q_2) xy \, \phi$.
\end{proof}

Thus, mt-logic is expressible enough to branch these quantifiers, and so
should be able to formalize a larger fragment of natural languages than
Dependence logic. However, this result is not completely satisfactory as the
formula that expresses branching is unnatural and we have not been able to
establish a result in which the branching $\text{Br}(Q_1,Q_2)x,y \phi
(x,y)$ can be expressed by a formula of the form $$ Q_1 x Q_2 y ( \psi
(x,y) \land \phi(x,y)).$$

\section{Conclusion}

The results in this paper demonstrates that it is possible to handle
generalized quantifiers within team semantics, given that the flatness
principle is relaxed. This opens up for a more general, and more algebraic,
definition of team semantics that quantifies over \emph{all} possible
``lifts'', in the sense that we don't restrict the possible semantic values
of atomic formulas. In the context of Dependence logic the semantic values of
atomic formulas are confined to principal ideals of teams. In the team
semantics used in this paper the restriction is to singleton sets. In a
related paper \cite{engstrom2023propositional} we investigate the case of
propositional logic when no such restriction is imposed.  The resulting
logic, known as the logic of teams, proves powerful enough to express
propositional dependence logic and all its related variants. Looking ahead,
we aim to merge the approach presented in this paper with that of the logic
of teams.

\section*{Acknowledgments}

This paper was written as part of the project: \emph{Foundations for team
semantics: Meaning in an enriched framework}, a research project supported by
grant 2022-01685 of the Swedish Research Council, Vetenskapsrådet.

%\bibliography{refs}
\printbibliography

\end{document}